\documentclass[12pt,reqno]{amsart}
\usepackage{a4wide}
\usepackage{amsmath,amsfonts,amssymb,amsthm,mathrsfs}
\usepackage{enumitem}
\usepackage[english]{babel}
\usepackage{color}

\setlist[enumerate]{label={\rm(\roman*)}}

\newtheorem{thm}{Theorem}[section]

\theoremstyle{remark}
\newtheorem{rem}[thm]{Remark}

\def\R{\mathbb R}
\def\RN{\R^N}

\renewcommand{\phi}{\varphi}

\newcommand{\dx}{\mathrm{\,d}x}
\newcommand{\dy}{\mathrm{\,d}y}

\newcommand{\irn}{\int_{\RN}}
\newcommand{\Lphi}{\mathscr L^\Phi}

\newcommand{\len}[1]{\left|{#1}\right|_N}
\newcommand{\lenn}[1]{\left|{#1}\right|_{2N}}

\newtoks\by
\newtoks\paper
\newtoks\book
\newtoks\jour
\newtoks\yr
\newtoks\pages
\newtoks\vol
\newtoks\publ
\newtoks\eds
\newtoks\proc
\def\ota{{\hbox{???}}}
\def\cLear{\by=\ota\paper=\ota\book=\ota\jour=\ota\yr=\ota
	\pages=\ota\vol=\ota\publ=\ota}
\def\endpaper{\textsc{\the\by}, \textit{\the\paper},
	{\the\jour} \textbf{\the\vol} (\the\yr), \the\pages.\cLear}
\def\endbook{\textsc{\the\by}, {\the\book}, \the\publ, \the\yr.\cLear}
\def\endprep{\textsc{\the\by}, \textit{\the\paper}, \the\jour.\cLear}
\def\endproc{\textsc{\the\by}, \textit{\the\paper}, \the\publ, \the\pages.\cLear}
\def\name#1#2{#1\,#2}
\def\et{ and }
\begin{document}

\title[A~weak-type expression of the Orlicz modular]{A~weak-type expression of the Orlicz modular}

\author[M.~K\v repela]{Martin K\v repela$^*$}
\address{Czech Technical University in Prague, Faculty of Electrical Engineering, Department of Mathematics, Technick\'a~2, 166~27 Praha~6, Czech Republic}
\email{martin.krepela@fel.cvut.cz}

\author[Z.~Mihula]{Zden\v ek Mihula$^{\ast\ast}$}
\address{Czech Technical University in Prague, Faculty of Electrical Engineering, Department of Mathematics, Technick\'a~2, 166~27 Praha~6, Czech Republic}
\email{mihulzde@fel.cvut.cz}

\author[J.~Soria]{Javier Soria$^\dagger$}
\address{Interdisciplinary Mathematics Institute (IMI), Department of Analysis and Applied Mathematics, Complutense University  of Madrid, 28040 Madrid, Spain}
\email{javier.soria@ucm.es}

\thanks{$^*$M. K\v repela was supported by the project OPVVV CAAS CZ.02.1.01/0.0/0.0/16\_019/0000778}
	 
\thanks{$^{\ast\ast}$Z. Mihula was supported by the project OPVVV CAAS CZ.02.1.01/0.0/0.0/16\_019/0000778 and by the grant GA\v CR P201/21-01976S}

\thanks{$^\dagger$J. Soria  was partially supported by grants PID2020-113048GB-I00 funded by MCIN/AEI/ 10.13039/501100011033, and Grupo UCM-970966}
	 
\begin{abstract}
An equivalent expression of Orlicz modulars in terms of measure of level sets of difference quotients is established. The result in a sense complements the famous Maz'ya--Shaposhnikova formula for the fractional Gagliardo--Slobodeckij seminorm and its recent extension to the setting of Orlicz functions.
\end{abstract}



\date{\today}

\maketitle

\section{Introduction}
%
%
%
	%
%
The fractional order Sobolev spaces $W^{s,p}(\RN)$, $p\in[1,\infty)$, $s\in(0,1)$, endowed with the Gagliardo--Slobodeckij seminorm, which is defined for
smooth compactly supported functions $u$ as
\begin{equation*}
|u|_{s,p}^p = \int_{\RN}\int_{\RN} \left(\frac{|u(x) - u(y)|}{|x-y|^s}\right)^p \frac1{|x-y|^N} \dx\dy,
\end{equation*}
have played an important role in the theory of partial differential equations and its applications for a long time (see the introductory section of \cite{Hitchhiker}). Much as it is tempting to think that
\begin{align}
\lim_{s\to1^-} |u|_{s,p}^p &\approx \int_{\RN} |\nabla u(x)|^p\dx \label{intro:s=1_wrong}\\
\intertext{or}
\lim_{s\to0^+} |u|_{s,p}^p &\approx \int_{\RN} |u(x)|^p\dx, \nonumber
\end{align}
the Gagliardo--Slobodeckij seminorm notoriously fails to capture these limiting cases\textemdash to that end, it is sufficient to consider any nonconstant $u\in\mathcal C_0^\infty(\RN)$ and observe that $|u|_{s,p}^p$ converges to $\infty$ as $s\to1^-$ or $s\to0^+$. Nevertheless, it was discovered around 20~years ago that these ``defects'' can be, in a sense, ``fixed'' by introducing certain compensatory factors. Namely, for every $u\in\mathcal C_0^\infty(\RN)$, a special case of what is now often called the Bourgain--Brezis--Mironescu formula \cite{BBM} tells us that

\begin{equation}
\lim_{s\to 1^+} (1-s) \int_{\RN}\int_{\RN} \left(\frac{|u(x) - u(y)|}{|x-y|^s}\right)^p \frac1{|x-y|^N} \dx\dy = C(N,p) \int_{\RN} |\nabla u(x)|^p\dx. \label{intro:BBM_formula}
\end{equation}

Moreover, V.G.~Maz'ya and T.~Shaposhnikova proved in \cite{MS} that

\begin{equation}
\lim_{s\to 0^+} s \int_{\RN}\int_{\RN} \left(\frac{|u(x) - u(y)|}{|x-y|^s}\right)^p \frac1{|x-y|^N} \dx\dy = C(N,p) \int_{\RN} |u(x)|^p \dx. \label{intro:MS_formula}
\end{equation}

Recently, a completely different approach, not involving integration of fractional difference quotients at all, to repairing \eqref{intro:s=1_wrong} was taken by H.~Brezis, J.~Van Schaftingen and P.-L.~Yung. They proved in \cite{BvSY} that, instead of introducing a compensatory factor, the limit as $s\to1^-$ can be recovered if the strong $L^p$ norm of fractional difference quotients is replaced by the weak $L^{p,\infty}$ quasi-norm. More precisely, they obtained the following result. 
Let $p\in[1,\infty)$ and $u\in\mathcal C_0^\infty(\RN)$ and define

\begin{equation*}
E_{\lambda,1} = \left|\left\{(x,y)\in \R^{2N}\colon x\neq y, \frac{|u(x) - u(y)|^p}{|x-y|^{N+p}}\geq \lambda^p\right\}\right|_{2N},
\end{equation*}
where $|\cdot|_{2N}$ stands for the Lebesgue measure on $\R^{2N}$. In \cite{BvSY} it was shown that

\begin{align*}
c(N,p)\int_{\RN} |\nabla u(x)|^p \dx \leq &\sup_{\lambda > 0} \lambda^p E_{\lambda,1} \leq C(N)\int_{\RN} |\nabla u(x)|^p\dx
\intertext{and}
\lim_{\lambda\to\infty} \lambda^p E_{\lambda,1} &= C(N,p) \int_{\RN} |\nabla u(x)|^p\dx.
\end{align*}
Following this innovatory approach, Q.~Gu and P.-L.~Yung established in \cite{GY} other, possibly even more unanticipated, formulae. They complement the Maz'ya--Shaposhnikova formula \eqref{intro:MS_formula} in the same way the result of Brezis, Van Schaftingen and Yung complements the Bourgain--Brezis--Mironescu formula \eqref{intro:BBM_formula}. Namely the result of \cite{GY} asserts that, for every $u\in L^p(\RN)$, $p\in[1, \infty)$,
\begin{align}
c(N)\int_{\RN} |u(x)|^p\dx \leq &\sup_{\lambda > 0} \lambda^p E_{\lambda,0} \leq C(N)\int_{\RN} |u(x)|^p\dx \label{intro:GU_Yung_equivalent_norm}
\intertext{and}
\lim_{\lambda\to\infty} \lambda^p E_{\lambda,0} &= C(N) \int_{\RN} |u(x)|^p\dx, \label{intro:GU_Yung_limit}
\end{align}
where

\begin{equation*}
E_{\lambda,0} = \left|\left\{(x,y)\in \R^{2N}\colon x\neq y, \frac{|u(x) - u(y)|^p}{|x-y|^{N}}\geq \lambda^p\right\}\right|_{2N}.
\end{equation*}

The classical results \eqref{intro:BBM_formula} and \eqref{intro:MS_formula} were recently considerably strengthened in \cite{ACPS0, ACPS1, FB-S} by replacing the $p$-th power in the integrals with Orlicz functions, thus allowing for non-polynomial growth.

The aim of this short paper is to similarly extend the new developments of \cite{GY}; i.e.~\eqref{intro:GU_Yung_equivalent_norm} and \eqref{intro:GU_Yung_limit}, by replacing the $p$-th power with a general Orlicz function globally satisfying the $\Delta_2$ condition. Therefore, we express the Orlicz modular in terms of measures of certain level sets, without using any integral. Our proof technique is based on the argument presented in \cite{GY}, appropriately extended to the Orlicz framework.


\medskip

In what follows, we introduce some basic notations and definitions, needed for understanding the setting we will be considering in our main result. A proper detailed treatment of Orlicz functions and classes may be found e.g.~in \cite{RR}. 

\medskip

A \emph{Young function} $\Phi\colon [0,\infty)\to[0,\infty)$ is any continuous convex function vanishing at $0$. Note that Young functions are nondecreasing. We say that a Young function $\Phi$ (globally) satisfies the \emph{$\Delta_2$ condition} if there is $k > 0$ such that $\Phi(2t)\leq k\Phi(t)$, for every $t>0$. Then necessarily $k\geq2$, which follows from the convexity of $\Phi$. We denote by $\Delta_2(\Phi)$ the infimum over all such $k$.

Given a Young function $\Phi$, we say that a measurable function $u\colon \RN\to\R$ belongs to the \emph{Orlicz class $\Lphi$}, and write $u\in\Lphi$, if
\begin{equation*}
\int_{\RN}\Phi(|u(x)|)\dx < \infty.
\end{equation*}
If $\Phi$ satisfies the $\Delta_2$ condition, $u\in\Lphi$ implies
\begin{equation*}
\int_{\RN}\Phi(\gamma|u(x)|)\dx < \infty, \qquad \text{for every $\gamma > 0$}.
\end{equation*}
As usual, $\omega_N$ denotes the volume of the unit ball in $\RN$.
	
\section{Main result}

\begin{thm}
	Let $\Phi$ be a Young function satisfying the $\Delta_2$ condition. Let $u\in\Lphi$ and for every $t>0$ define
		$$
			E_t = \left\{ (x,y)\in\R^{2N}\colon\ x\ne y,\ \frac{\Phi(|u(x)-u(y)|)}{|x-y|^N} \ge \Phi(t) \right\}.
		$$
	Then
		\begin{equation}\label{2}
			2\omega_N \irn \Phi(|u(x)|)\dx = \lim_{t\to 0^+} \Phi(t)\,\lenn{E_t}.
		\end{equation}
	Furthermore,
	\begin{equation}\label{main:equivalent_expression_modular}
	2\omega_N \irn \Phi(|u(x)|)\dx\leq \sup_{t>0} \Phi(t)\,\lenn{E_t} \leq 2\omega_N \Delta_2(\Phi) \irn \Phi(|u(x)|)\dx.
	\end{equation}
\end{thm}

\begin{proof}
	For every $t>0$ define the set
		$$
			H_t = \{ (x,y)\in E_t\colon\ |y|>|x|\}
		$$
	and observe that, thanks to symmetry, it satisfies $\lenn{H_t}=\frac12 \lenn{E_t}$.
	
	At first, we are going to suppose that $u$ has compact support; i.e.,~there exists $R>0$ such that
		$$
			\operatorname{supp} u \subset B_R.
		$$
	Notice that, if $(x,y)\in H_t$, then necessarily $x\in B_R$, otherwise we would have $x,y\in \R^N\setminus B_R$ and thus
	$u(x)=u(y)=0$, which would imply $(x,y)\notin H_t$.
	
	For a fixed $x\in B_R$ define the sets
		$$
			H_{t,x}  = \{ y\in\RN\colon\ (x,y)\in H_t\}  = \left\{ y\in\RN \colon |y|>|x|,\ \frac{\Phi(|u(x)-u(y)|)}{|x-y|^N} \ge \Phi(t) \right\}
		$$
	and
		$$
			H_{t,x,R}  = H_{t,x}\setminus B_R = \left\{ y\in\RN \colon |y|>R,\ \frac{\Phi(|u(x)|)}{|x-y|^N} \ge \Phi(t) \right\}.
		$$
	Obviously, we have
		\begin{equation}\label{3}
			H_{t,x,R} \subset H_{t,x} \subset H_{t,x,R}\cup B_R.
		\end{equation}
	The first inclusion together with the definition of $H_{t,x,R}$ implies
		\begin{equation}\label{4}
			\len{H_{t,x}}\ge \len{H_{t,x,R}} \ge \omega_N\frac{\Phi(|u(x)|)}{\Phi(t)} - \omega_NR^N,
		\end{equation}
	while the second inclusion in \eqref{3} implies
		\begin{equation}\label{5}
			\len{H_{t,x}} \le \omega_N\frac{\Phi(|u(x)|)}{\Phi(t)} + \omega_NR^N.
		\end{equation}
	Since $x\in B_R$ was arbitrarily chosen, we may integrate \eqref{4} and \eqref{5} over $B_R$ with respect to $x$ and multiply by $\Phi(t)$ to get
		\begin{align*}
			\omega_N \int_{B_R} \Phi(|u(x)|) \dx - \Phi(t)\omega_N^2R^{2N} & \le \Phi(t)\int_{B_R}\len{H_{t,x}} \dx \\
			& \le \omega_N \int_{B_R} \Phi(|u(x)|) \dx + \Phi(t)\omega_N^2R^{2N}.
		\end{align*}
	Recalling that $u$ is supported in $B_R$ and $\lenn{H_t}=\frac12 \lenn{E_t}$, we may further rewrite this as
		\begin{align}
			2\omega_N \irn \Phi(|u(x)|) \dx - 2\Phi(t)\omega_N^2R^{2N} & \le  \Phi(t)\lenn{E_t} \notag\\
			& \le 2\omega_N \irn \Phi(|u(x)|) \dx + 2\Phi(t)\omega_N^2R^{2N}.\label{8}
		\end{align}
	Letting $t\to 0^+$, we obtain \eqref{2}. 
	
	Now we are going to extend the result beyond compactly supported functions. Suppose that $u\colon\R^N\to \R$ is measurable. 
	For any fixed $t>0$, the set $E_t$ satisfies
		\begin{equation}\label{6}
			E_t \subset \left\{ (x,y)\in \R^{2N}\colon\ \frac{\Phi(2|u(x)|)}{|x-y|^N} \ge \Phi(t)\right\} 
			\cup \left\{ (x,y)\in \R^{2N}\colon\ \frac{\Phi(2|u(y)|)}{|x-y|^N} \ge \Phi(t)\right\}.
		\end{equation}
	Indeed, if $(x,y)\in\R^N$ is not contained in either of the two sets on the right-hand side, then, by monotonicity and convexity of $\Phi$,
		\begin{align*}
			\Phi(|u(x)-u(y)|) & \le \Phi(|u(x)|+|u(y)|)\\
								& \le \frac12 \Phi\left(2|u(x)|\right) + \frac12\Phi\left(2|u(y)|\right) \\
								& < \Phi(t)|x-y|^N,
		\end{align*}
	hence $(x,y)\notin E_t$. This shows \eqref{6}.
	
	Using the symmetry of the two sets on the right-hand side of \eqref{6}, we obtain
		\begin{align*}
			\lenn{E_t}   \le 2 \irn \irn \chi_{\left\{ (x,y)\in\R^{2N}\colon\ |x-y|^N\le \frac{\Phi(2|u(x)|)}{\Phi(t)}\right\}}(x,y) \dy \dx 
			 = 2\frac{\omega_N}{\Phi(t)} \irn \Phi(2|u(x)|) \dx,
		\end{align*}
	hence
		\begin{equation}\label{7}
			\sup_{t>0}\Phi(t)\lenn{E_t} \le 2 \omega_N \irn \Phi(2|u(x)|) \dx.
		\end{equation}
	Notice that 
	neither the assumption 
	nor the $\Delta_2$ condition of $\Phi$ has been used yet, so this estimate in fact holds for any measurable $u$ and any Young function $\Phi$. If $\Phi$ satisfies the $\Delta_2$ condition, \eqref{7} readily implies the second inequality in \eqref{main:equivalent_expression_modular}.
		
	Assume that $u\in \Lphi$. Choose $R>0$ and define
		\begin{equation}\label{14}
			u_R = u\chi_{B_R}, \qquad v_R = u - u_R.
		\end{equation}
	Furthermore, choose $t>0$, $\lambda\in(0,1)$ and define
		\begin{align*}
			A_1 & = \left\{ (x,y)\in \R^{2N}\colon\ \frac{\Phi\left(\frac1\lambda|u_R(x)-u_R(y)|\right)}{|x-y|^N} \ge \Phi(t)\right\}
			\intertext{and}
			A_2 & = \left\{ (x,y)\in \R^{2N}\colon\ \frac{\Phi\left(\frac1{1-\lambda}|v_R(x)-v_R(y)|\right)}{|x-y|^N} \ge \Phi(t)\right\}.
		\end{align*}
	Then $E_t\subset A_1 \cup A_2$. Similarly as before, this can be seen by using monotonicity and convexity of $\Phi$ to get
		\begin{align*}
			\Phi(|u(x)-u(y)|) & \le \Phi(|u_R(x)-u_R(y)| + |v_R(x)-v_R(y)|) \\
			& \le \lambda \Phi\left(\frac1\lambda|u_R(x)-u_R(y)|\right) + (1-\lambda) \Phi\left(\frac1{1-\lambda}|v_R(x)-v_R(y)|\right),
		\end{align*}
	from which the inclusion follows. 	
	
	Observe that the set $A_1$ is obtained by replacing $u$ with $\frac{u_R}\lambda$ in the definition of $E_t$. Since $\frac{u_R}\lambda$ is compactly supported and belongs to $\Lphi$ (since $\Phi$  satisfies the $\Delta_2$ condition),
	we may use the previously obtained estimate \eqref{8} with $A_1$ and $\frac{u_R}\lambda$ in place of $E_t$ and $u$, respectively, to get
		\begin{equation*}
			\Phi(t)\lenn{A_1} \le 2\omega_N \irn \Phi\left(\frac{|u_R(x)|}\lambda\right) \dx + 2\Phi(t)\omega_N^2R^{2N}.
		\end{equation*}
	Analogously, applying \eqref{7} to the function $\frac{v_R}{1-\lambda}$ in place of $u$ ($A_2$ plays the role of $E_t$ for this function 
	), we get
		\begin{equation*}
			\Phi(t)\lenn{A_2} \le 2 \omega_N \irn \Phi\left(\frac{2|v_R(x)|}{1-\lambda}\right) \dx. 
		\end{equation*}
	As $E_t\subset A_1 \cup A_2$, this gives
		\begin{equation*}
			\Phi(t)\lenn{E_t} \le 2\omega_N \left[ \irn \Phi\left(\frac{|u_R(x)|}\lambda\right) \dx + \Phi(t)\omega_NR^{2N} + \irn \Phi\left(\frac{|2v_R(x)|}{1-\lambda}\right) \dx\right].
		\end{equation*}
	Since $|u_R|\le |u|$, $|v_R|\le |u|$, $u\in\Lphi$ and $\Phi$ satisfies the $\Delta_2$ condition, both integrals above are finite regardless the choice of $\lambda$,
	and the second integral (with a fixed $\lambda$) vanishes as $R\to\infty$ by the dominated convergence theorem. Hence, consecutively letting $t\to 0^+$, $R\to \infty$ and $\lambda\to 1^-$, we finally obtain
		\begin{equation}\label{12}
			\limsup_{t\to 0^+} \Phi(t)\lenn{E_t} \le 2\omega_N \irn \Phi(|u(x)|) \dx.
		\end{equation}
	
	It remains to show the opposite inequality for the lower limit. 
	Fix $R>0$, $\lambda\in(0,1)$ and define $u_R$, $v_R$ as in \eqref{14}. Then $|u|-|v_R|=|u_R|$ and, by convexity of $\Phi$, for any $(x,y)\in\R^{2N}$ we have
		\begin{equation}\label{13}
			\frac1\lambda \Phi(\lambda|u_R(x)-u_R(y)|) - \frac{1-\lambda}\lambda \Phi\left( \frac{\lambda}{1-\lambda}|v_R(x)-v_R(y)|\right) \le \Phi(|u_R(x)-u_R(y)|) .
		\end{equation}
	For any $t>0$ define
		\begin{align*}
			A_3 & = \left\{ (x,y)\in \R^{2N}\colon\ \frac{\Phi\left(\lambda|u_R(x)-u_R(y)|\right)}{|x-y|^N} \ge \Phi(t)\right\}
			\intertext{and}
			A_4 & = \left\{ (x,y)\in \R^{2N}\colon\ \frac{\Phi\left(\frac\lambda{1-\lambda}|v_R(x)-v_R(y)|\right)}{|x-y|^N} \ge \Phi(t)\right\}.
		\end{align*}	
	Whenever $(x,y)\in A_3\setminus A_4$, we have
		\begin{align*}
			\Phi(t) & = \frac1\lambda \Phi(t) - \frac{1-\lambda}\lambda \Phi(t)\\
			& < \frac1{|x-y|^N} \left( \frac1\lambda \Phi(\lambda|u_R(x)-u_R(y)|) - \frac{1-\lambda}\lambda \Phi\left( \frac{\lambda}{1-\lambda}|v_R(x)-v_R(y)|\right) \right) \\
			& \le \frac{\Phi(|u_R(x)-u_R(y)|)}{|x-y|^N},
		\end{align*}		
	where the last inequality follows from \eqref{13}. This shows that $(x,y)\in E_t$. Thus, we have $E_t \supset A_3\setminus A_4$.
	
	We proceed analogously as before, realizing that $A_3$ plays the role of $E_t$ for the compactly supported function $\lambda u_R$, we use \eqref{8}	to obtain
		$$
			\Phi(t)\lenn{A_3} \ge 2\omega_N \left[ \irn \Phi(\lambda|u_R(x)|)\dx - \Phi(t)\omega_NR^{2N}\right].
		$$
	Similarly, an appropriate interpretation of \eqref{7} yields
		$$
			\Phi(t)\lenn{A_4} \le 2\omega_N \irn \Phi\left(\frac{2\lambda}{1-\lambda}|v_R(x)|\right)\dx.
		$$
	Hence,
		\begin{align*}
			\Phi(t)\lenn{E_t} & \ge \Phi(t)\left(\lenn{A_3} - \lenn{A_4}\right) \\
			& \ge 2\omega_N \left[ \irn \Phi\left(\frac{|u_R(x)|}\lambda\right) \dx - \Phi(t)\omega_NR^{2N} - \irn \Phi\left(\frac{2|v_R(x)|}{1-\lambda}\right) \dx\right].
		\end{align*}
	Letting $t\to 0^+$, $R\to \infty$ and $\lambda\to 1^-$, in this order, now yields
		\begin{equation}\label{15}
			\liminf_{t\to 0^+} \Phi(t)\lenn{E_t} \ge 2\omega_N \irn \Phi(|u(x)|) \dx.	
		\end{equation} 
	Once again, the $\Delta_2$ condition of $\Phi$ as well as the assumption $u\in\Lphi$ are both required in this step. This clearly implies the first inequality in \eqref{main:equivalent_expression_modular}. Finally, combining \eqref{15} with \eqref{12}, we	arrive at \eqref{2}, and so the proof is complete.	
\end{proof}

\begin{rem}
Applying the theorem to $\Phi(t)=t^p$, $p\in[1,\infty)$, we recover \cite[Theorem~1]{GY} with the same multiplicative constants.
\end{rem}


\begin{thebibliography}{99}
\bibitem{ACPS0} 
	\by{\name{A.}{Alberico}, \name{A.}{Cianchi}, \name{L.}{Pick}\et\name{L.}{Slav\'\i kov\'a}}
	\paper{On the limit as {$s\to 0^+$} of fractional Orlicz-Sobolev spaces}
	\jour{J.~Fourier Anal.~Appl.}
	\vol{26:6}
	\yr{2020}
	\pages{Paper No.~80}
	\endpaper
	
	\bibitem{ACPS1} 
	\by{\name{A.}{Alberico}, \name{A.}{Cianchi}, \name{L.}{Pick}\et\name{L.}{Slav\'\i kov\'a}}
	\paper{On the limit as {$s \to 1^-$} of possibly non-separable fractional {O}rlicz-{S}obolev spaces}
	\jour{Atti Accad.~Naz.~Lincei Rend.~Lincei Mat.~Appl.}
	\vol{31:4}
	\yr{2020}
	\pages{879--899}
	\endpaper



\bibitem{BBM} 
	\by{\name{J.}{Bourgain}, \name{H.}{Brezis}\et\name{P.}{Mironescu}}
	\paper{Another look at Sobolev spaces}
	\jour{in: Optimal Control and Partial Differential Equations, IOS, Amsterdam}
	\vol{\!\!}
	\yr{2001}
	\pages{439--455}
	\endpaper
	
	\bibitem{BvSY} 
	\by{\name{H.}{Brezis}, \name{J.}{Van Schaftingen}\et\name{P.-L.}{Yung}}
	\paper{A~surprising formula for Sobolev norms}
	\jour{Proc.\ Natl.\ Acad.\ Sci.}
	\vol{118:8}
	\yr{2021}
	\pages{e2025254118}
	\endpaper
	
	\bibitem{Hitchhiker} 
	\by{\name{E.}{Di Nezza}, \name{G.}{Palatucci}\et\name{E.}{Valdinoci}}
	\paper{Hitchhiker's guide to the fractional {S}obolev spaces}
	\jour{Bull.~Sci.~Math.}
	\vol{136:5}
	\yr{2012}
	\pages{521--573}
	\endpaper
	
	\bibitem{FB-S} 
	\by{\name{J.}{Fern\'andez Bonder},\et \name{A.~M.}{Salort}}
	\paper{Fractional order Orlicz-Sobolev spaces}
	\jour{J.~Funct.~Anal.}
	\vol{277:2}
	\yr{2019}
	\pages{333--367}
	\endpaper
	

 
	\bibitem{GY} 
	\by{\name{Q.}{Gu}\et\name{P.-L.}{Yung}}
	\paper{A new formula for the {$L^p$} norm}
	\jour{J.~Funct.\ Anal.}
	\vol{281}
	\yr{2021}
	\pages{Paper No.~109075, 19~pp}
	\endpaper

	
	\bibitem{MS} 
	\by{\name{V.~G.}{Maz'ya}\et\name{T.}{Shaposhnikova}}
	\paper{On the Bourgain, Brezis, and Mironescu theorem concerning limiting embeddings of fractional Sobolev spaces}
	\jour{J.~Funct.~Anal.}
	\vol{195:2}
	\yr{2002}
	\pages{230--238}
	\endpaper
	
	
	\bibitem{RR} 
	\by{\name{M.~M.}{Rao}\et\name{Z.~D.}{Ren}}
	\book{Theory of {O}rlicz spaces. Monographs and Textbooks in Pure and Applied Mathematics~146}
	\yr{1991}
	\pages{xii+449}
	\publ{Marcel Dekker, Inc., New York}
	\endbook
	
	
\end{thebibliography}
\end{document}